\documentclass[]{article}

\usepackage{moreverb}

\usepackage{amsmath}
\usepackage{authblk}
\usepackage{amssymb}
\usepackage{color}

\newtheorem{theorem}{Theorem}[section]
\newtheorem{definition}[theorem]{Definition}
\newtheorem{lemma}[theorem]{Lemma}

\newtheorem{remark}[theorem]{Remark}
\newtheorem{proof}[theorem]{Proof}

\def\RR{\mathbb{R}}
\begin{document}

%\runninghead{J. Cordero and A. Jim\'enez}

\title{Well posedness and stationary solutions of a neural field equation with synaptic plasticity}
\author[1]{Juan Cordero Ceballos \thanks{ jccorderoc@unal.edu.co}}
\author[2]{Alejandro Jim\'enez Rodr\'iguez\thanks{ajimenezrodriguez1@sheffield.ac.uk}}

\affil[1]{Departamento de Matem\'aticas y Estad\'istica, Universidad Nacional de Colombia-Sede Manizales, Colombia.}
\affil[2]{Department of Psychology, University of Sheffield, UK}
%\author{Juan Cordero\affil{a}\corrauth\ and
%Alejandro Jim\'enez\affil{b}}

%\address{\affilnum{a}Departamento de Matem\'aticas y Estad\'istica, Universidad Nacional de Colombia-Sede Manizales, Colombia.
%\\
%\affilnum{b}Second author's address}

%\corraddr{Departamento de Matem\'aticas y Estad\'istica, Universidad Nacional de Colombia-Sede Manizales, Colombia.
%E-mail: jccorderoc@unal.edu.co}

%\cgs{<Contract/grant sponsor name>}
\maketitle

\begin{abstract}
We consider the initial value problem associated to the neural field equation of Amari type with plasticity
\[
 u_t(x,t)=-u(x,t)+\int_{\Omega}w(x,y)[1+\gamma g( u(x,t) - u(y,t) )] f(u(y,t))\; dy, \;(x,t) \in \Omega \times (0, \infty),
\]
where $\Omega\subset\mathbb{R}^m$, $f$ and $g$ are bounded and continuously differentiable functions with bounded derivative, and $\gamma\ge0$ is the plasticity synaptic coefficient. We show that the problem is well posed in $C_b(\mathbb{R}^m)$ and $L^1(\Omega)$ with $\Omega$ compact. The proof follows from a classical fixed point argument when we consider the equation's flow. Strong convergence of solutions in the no plasticity limit ($\gamma\to0$) to solutions of Amari's equation is analysed. Finally, we prove existence of stationary solutions in a general way. As a particular case, we show that the Amari's model, after learning, leads to the stationary Schr\"odinger equation for a type of gain modulation.   

\end{abstract}

%\MOS{<Subject classification numbers>}

%\keywords{Neural field equations; Amari equation; Synaptic Plasticity; Gain modulation; Well posedness; Stationary Schr\"odinger equation.}

\section{Introduction}
\label{intro}

Neural fields equations are models of the activity of spatially structured populations of neurons in the continuum limit  \cite{coombes2014neural}. In these models, the population of neurons is represented by a spatial domain $\Omega \subset \mathbb{R}^m$ and the activity at position $x \in \Omega$ corresponds to a time varying scalar field $u(x,t)$. The dynamics of the population is described, in general, by an integrodifferential equation. Multiple interacting populations are thus described by coupled systems of equations. The ultimate form of the equation that describes the dynamics of the field depends on the kind of phenomena being modeled, which range from single neuron properties (e.g. refractory periods \cite{wilson1972excitatory}, membrane dynamics \cite{coombes2014neural}, adaptation\cite{coombes2005bumps}) to synaptic or population features (e.g. transmission delays \cite{kao2016absolute}, plasticity\cite{abbott2000synaptic, abbassian2012neural, fotouhi2015continuous}, inhibitory/excitatory connectivity\cite{wilson1972excitatory})\\

Neural fields have been used to model EEG rhythms, epilepsy, binoculary rivalry and a broad range of cognitive phenomena\cite{coombes2014neural, bressloff2011spatiotemporal}. A recent account of the research and challenges in the area can be found in \cite{coombes2014neural}. Two commonly used neural fields models are the Wilson and Cowan  model \cite{wilson1972excitatory} and the Amari neural field equation \cite{amari1977dynamics}. In the Wilson and Cowan model, the field variable corresponds to the proportion of firing neurons. This model includes refractory periods and nonlocal interations which are usually removed  in modern mean field theories.  The Amari model can be voltage based (i.e. the field variable represents the mebrane potential) or rate based (i.e. the field variable represents the firing rate); this model can be seen as a special case of the  Wilson and Cowan one in which the refractory period is removed but nonlocal interactions are preserved.\\

In this work we focus in a voltage based neural field model of Amari type with plasticity. In the traditional voltage based single population Amari model\cite{amari1977dynamics}, the dynamics of the membrane voltage $u(x,t)$ is described by the following nonlinear integrodifferential equation of Hammerstein type

\begin{align}\label{E:Amari}
 &u_t(x,t)+u(x,t)=\int_{\mathbb{R}^m}\omega(x,y)f(u(y,t))dy.
\end{align}

The left hand part of the equation represents a first order membrane dynamics (exponential decay of incoming pulses). The right hand side represents the synaptic inputs from other neurons in the population. The kernel of the intergral $w(x,y)$ is interpreted as the weight or strength between the \emph{pre-synaptic} neuron at position $y$ and the \emph{post-synaptic} one at position $x$, and it is called \emph{synaptic kernel}. The sign of the kernel determines whether the synapse is excitatory or inhnibitory. Commonly used kernels are, the purely excitatory $w(\delta)=\exp(-|\delta|)/2$ and the so called mexican hat $w(\delta ) = (1 - |\delta |)\exp(-|\delta|)$, which represents long range inhibitiory - short range excitatiory connections\cite{ermentrout1998neural, coombes2014neural}; it is usually assumed $w(x,y) = w(|x-y|)$ (i.e. isotropic connectivity). $f(u(x,t))$ stands for the firing rate of the neuron as a function of the membrane voltage; it can be understood as the probability of neuron $x$ firing an action potential at time $t$ . Commonly used firing rate functions are sigmoidal functions like $f(\delta) = 1/(1 + exp(-\delta))$ or $f(\delta) = \arctan(\delta)$, or linear/piece-wise linear functions\cite{ bressloff2011spatiotemporal,ermentrout1998neural}.\\

Plasticity is introduced here by modifying the kernel in an Hebbian, activity dependent manner (i.e, neurons that fire together, increase their connection strength)\cite{gerstner2002mathematical, abbassian2012neural, fotouhi2015continuous}; that is, the term plasticity is understood in a broad sense that include many synaptic mechanisms (e.g. facilitation, potentiation, depression, etc.). In this paper we study the particular dynamic kernel model of Abbasian et al. \cite{abbassian2012neural}
\begin{equation}\label{E:Abbassian}
u_t(x,t)+u(x,t)=\int_{-\infty}^\infty\omega_{g,\gamma}(x,y)f(u(y,t))dy,
\end{equation}
with
\begin{equation}\label{E:Abbassian2}
\omega_{g,\gamma}(x,y) = \omega(x,y)[1 + \gamma g( u(x,t) - u(y,t) )],
\end{equation} 
where $g$, called the \emph{plasticity or learning kernel}, is the gaussian fuction $g(\delta) = e^{-\delta^2}$ and $\gamma\ge0$ which  we will refer to as the \emph{synaptic plasticity coefficient}. It is important to note that, in this model, $u(x,t)$ is the membrane potential as previously stated and, therefore, the voltage and not the firing rate is used to modify the synaptic strength.\\

The parameter $\gamma$ represents the increase in the synapstic strength when the pre and post-synaptic neurons fire together at the same rate ($g(u(x,t) - u(y,t)) \approx g(0) = 1$). This coefficient is related to the learning rate in discrete neural models\cite{gerstner2002mathematical}. When the membrane potentials $u(x,t)$ and $u(y,t)$ are similar, the connection of the neurons will be strengthened; when different, the activity is assumed to be uncorrelated and the behavior approaches that of the Amari's model ($g(\delta) \to 0$ when $\delta \to \infty$). Finally, when both neurons are silent, the functional form of the firing rate function have to be choosen in order avoid any unexpected behavior due to false increases in the synaptic strength.\\

Due to the particular form of the plasticity kernel, after the learning has ocurred (i.e. in the stationary case, $t \to \infty$), it can be divided in a pre-synaptic and post-synaptic contribution $g(u_\infty(x) - u_\infty(y)) = \phi_{\mathrm{post}}(x)\phi_{\mathrm{pre}}(y)$; as seen in section \ref{gainfield}. The pre-synaptic contribution defines a \emph{multiplicative gain field}, that is, a multiplicative modulation in the slope of the firing rate function in a position dependent way. Gain modulation is an important phenomenon related with neural computations \cite{chance2002gain, salinas2000gain}. It has been proposed to play a major role in attention, propagation of epileptic seizures and the stability of the underlying networks\cite{chance2011gain, stead2010microseizures}. Gain modulation emerges here as a by-product of learning, therefore, the function $\phi_{\mathrm{pre}}(y)$ represents the increase in the neurotransmiter release probability and  $\phi_{\mathrm{post}}(x)$ represents the increase in receptor density in the post-synaptic neuron (among other membrane and internal changes in the cell); both observed products of learning in biological neurons \cite{abbott2000synaptic}\\

To study the effects of the learned patterns on the dynamics of neural fields, the learning is deactivated but the stationary kernel is preserved and decomposed as previously described. The new kernel can be used in a second neural field equation in which the learned structured does not change any more but is fixed to the stationary solution of \eqref{E:Abbassian}
\begin{equation}
v_t(x,t)+v(x,t)=\frac{1}{2\lambda}\int_{\Omega}w(x,y)G(u_\infty(x) - u_\infty(y))f(v(y,t))dy.
\label{eq:amarigain2}
\end{equation}
We introduce a particular form of this equation in connection to the plasticity model being studied
\begin{equation}
v_t(x,t)+v(x,t)=\int_{\Omega}e^{-\lambda|x-y|}\phi_{\mathrm{pre}}(y)v(y,t)dy
\label{eq:amarigain}
\end{equation}
with $\omega(x)=\frac{1}{\lambda}e^{-\lambda|x|}$ and $\lambda>0$. In section \ref{gainfield} we study some of its properties and those results are extended in an upcomming paper \cite{Jimenez2017}.\\

Most of the theoretical research on neural fields have been focused in studying the properties of particular soutions, like resting state, bumps or traveling waves \cite{coombes2005bumps,coombes2014neural, bressloff2011spatiotemporal}. The work of Abbassian et al. is also oriented in that direction \cite{abbassian2012neural}. To our knowledge, the first rigorous approach to the existence of general solutions of the Amari model \eqref{E:Amari} was done by Potthast and Graben \cite{potthast2010existence}. Similar and further results can be found in \cite{faugeras2008absolute, oleynik2013properties}. Further analysis, like asymptotic behavior in bounded and unbounded domains can be found in \cite{da2012properties, da2014asymptotic}. For a smooth firing rate $f$, a well defined synaptic kernel $\omega$ and initial datum $u(x,0)=u_0(x), u_0\in C_b(\mathbb{R}^n)$ (the space of continous bounded functions),  in \cite{potthast2010existence} they showed that there exist a function  $v\in C([0,+\infty):C_b(\mathbb{R}^n))$, solution of \eqref{E:Amari} in $\mathbb{R}^n\times(0,+\infty)$ and  $v(x,0)=u_0(x)$. This is a global result for the Cauchy probleam associated to \eqref{E:Amari}. If $f$ is discontinous (e.g. a Heaviside function), the showed that this method is not valid. Then, with another kind of restrictions on the kernel and more specialized spaces, they obtained weaker solutions usign a compactness argument.\\

In this work we study the Initial Value Problem (IVP) associated with  \eqref{E:Abbassian}
\begin{align}
&u_t(x,t)+u(x,t)=\int_{\Omega}\omega_{g,\gamma}(x,y)f(u(y,t))dy,\;t>0,x\in\Omega,\label{E:PVI1}\\
&u(x,0)=u_0(x),\;x\in\Omega\subset\mathbb{R}^m,\label{E:PVI2} 
\end{align}
with $m\ge 1$. We follow the ideas of Potthast and Graben \cite{potthast2010existence} to show that the IVP \eqref{E:PVI1}-\eqref{E:PVI2} is globally well posed in $C_b(\mathbb{R}^n)$ and $L^1(\Omega)$ for $\Omega$ bounded. We also show that the solutions fot he Amari model  \eqref{E:Amari} are a good approximation for the model with plasticity of Abbasian et al. (an assumption used for those author in proving their results). More precisely, we show that solutions of \eqref{E:PVI1} converge uniformly to solutions of  \eqref{E:Amari} when $\gamma\to0$. In the last section we improve the results of \cite{abbassian2012neural} by showing the existence of stationary solutions of \eqref{E:PVI1} for small $\gamma$ and $\Omega\subset\mathbb{R}^m$ compact. Finally, the effect of learned gain fields is studied by introducing the Gain Field Equation \eqref{eq:amarigain} and we show that, for certain gain fields, it corresponds to the stationary Schr\"odinger equation. 

\section{Preliminary}
We begin this section defining the concept of well posedness of a Initial Value Problem (IVP).\\

Consider the following Cauchy problem:

\begin{equation}
\label{eq:cauchy}
\left\{
\begin{array}{l l}
u_t(t)&=F(t,u(t)) \in X,\\
u(0)&=u_0 \in Y,
\end{array} \right.
\end{equation}
where $X$, $Y$ are Banach spaces, $T_0 \in (0,\infty)$ and let $F:[0,T_0]  \times Y \to X$ is a continous function and $u_0$ is a given data.

\begin{definition}[Well posedness]
Let $X$, $Y$ be Banach spaces, $T_0 \in (0,\infty)$ and let $F:[0,T_0]  \times Y \rightarrow X$ be a continous function. We say that the Cauchy problem \eqref{eq:cauchy} is \textbf{locally well-posed in} $Y$ if:

 \begin{enumerate}
  \item There exist $T \in (0,T_0]$ and a function $u \in C([0,T];Y)$ such that $u(0) = u_0$ and the differential equation is satisfied,
  \item the problem (\ref{eq:cauchy}) has at most one solution in $ C([0,T];Y)$,
  \item the map $u_0 \longmapsto u$ is continuous.
\end{enumerate}
If the above conditions are satisfied for all $T \in [0,\infty)$, we say that the Cauchy problem is \textbf{globally well-posed} in $Y$.
\end{definition}

Next we present the appropriate spaces for the well posedness results, and we define the kernel considered by Potthas and Graben\cite{potthast2010existence}, that significantly extend the others one.
\begin{definition}
If  $M$ is a metric space and $I\subset\mathbb{R}$ a interval then
\begin{itemize}
\item[i)]  $C_{b}(M):=\{h:M\longrightarrow\RR\;|\;\text{$h$ is continuous and bounded}\}$, endowed with the norm 
\[
\|h\|:=\sup_{x\in{M}}|h(x)|.
\]
\item [ii)] $C^1_{b}(I):=\{g:I\longrightarrow\RR\;|\; g\in C^1(I)\,\text{and}\; g'\in C_b(I)\}$, with the norm 
\[
\|g\|:=\sup_{t\in{I}}(|g(t)|+|g'(t)|).
\]
\item [iii)] $C^{0,1}_{b}(\mathbb{R}^m\times[0,\infty)):=\{f:\mathbb{R}^m\times[0,\infty)\longrightarrow\RR\;|\;\forall t>0\;f(\cdot, t)\in C_b(\mathbb{R}^m),\; \forall x\in\mathbb{R}^m\; f(x,\cdot)\in C^1_b[0,\infty)\}$, and norm 
\[
\|f\|:=\sup_{x\in\mathbb{R}^m, t\ge0}\Big(|f(x,t)|+\Big\vert\frac{df}{dt}(x,t)\Big\vert\Big).
\]
\end{itemize}
\end{definition}

\begin{definition}{(Synaptic kernel)}
The synaptic integral kernel is a function $w:\RR^{m}\times\RR^{m}\to\mathbb{R}$ such that 
\begin{equation}
w \in L^{\infty}(\RR^{m}\times\RR^{m}),\quad\|w\|_{L^ {\infty}}=\sup_{x,y\in\RR^{m}}|w(x,y)|\le{C_{\infty}},\label{e9}
\end{equation}
and
\begin{equation}
 \|w\|_{L^{\infty}_xL^{1}_y}:=\sup_{x\in\RR^{m}}\|w(x,\cdot)\|_{L^{1}_y}\le{C_w},  \label{e8}
\end{equation}
that is
\begin{equation}
 w\in{L^{\infty}_xL^{1}_y(\RR^{m}\times\RR^{m})},\label{e7}
\end{equation}
\\
as a function $x\mapsto w(x,\cdot)$ of $L^{\infty}_x(\mathbb{R}^m)$ with $w(x,\cdot)\in L^1_y(\mathbb{R}^m)$  for all $x$, and for some constants $C_{\infty}, C_w>0$.\\
Moreover, $w$ satisfy the Lipschitz condition
\begin{equation}
 \|w(x,\cdot)-w(\widetilde{x},\cdot)\|_{_{L^{1}}}\le{K_w}|x-\widetilde{x}|\;\;\;\;\;x,\;\widetilde{x}\in\RR^{m}\label{e9}
\end{equation}
for some constant $K_w>0$.
\end{definition}

\begin{definition}
\label{def:learning}
 The activation function $f:\mathbb{R}\to\mathbb{R}$ and the modulation learning function $g:\mathbb{R}\to\mathbb{R}$ are those such that
 \begin{equation}
  f,g\in C_b^1(\mathbb{R}), 0\le f,g\le 1.
 \end{equation}
\end{definition}

The equation \eqref{E:Abbassian} suggest the use of the following operators:
\begin{align}
\label{eq:op1}
&(Ju)(x,t):= \int_{\Omega}{\omega_{g,\gamma}(x,y)f(u(y,t))\;dy},
\end{align}
\begin{align}
\label{eq:op2}
&(Fu)(x,t) := -u(x,t) + (Ju)(x,t),
\end{align}
and
\begin{align}
\label{eq:op3}
&(Au)(x,t) := \int_0^t{(Fu)(x,s) ds}.
\end{align}
Note that $J$ is the nonlinear part of equation \eqref{E:Abbassian} and we can write the IVP \eqref{E:PVI1}-\eqref{E:PVI2} as:
\begin{equation}
\label{eq:opkerdyn}
\left\{ 
  \begin{array}{l l}
    u' = Fu,\\
    u(x,0) = u_0(x),
  \end{array} \right.
\end{equation}
where $u' \equiv \frac{\partial u}{\partial t}$ and, by integration 
\begin{equation}
\label{eq:nfvolterra}
u = u_0 + Au.
\end{equation}

It is clear that a solution of \eqref{eq:nfvolterra} is a fixed point of the operator defined by $\varphi\to u_0+A\varphi$.

\begin{definition}
A solution of the IVP \eqref{eq:opkerdyn} or \eqref{eq:nfvolterra} is a function $u \in C(\RR^m \times [0,\rho))$, for some $\rho$, such that $u$ satisfies \eqref{eq:opkerdyn} or \eqref{eq:nfvolterra}. The flow of the equation $u'=Fu$ is the map 
\[
 (t,u_0)\to \varphi(t,u_0)
\]
defined as
\begin{equation}
 \varphi (t,u_0)=u_0+\int_0^t{(F\varphi(\cdot,u_0))(\cdot_{x},s)} ds, 
\end{equation}
where
\begin{equation}
 \varphi(\cdot,u_0)(x,s):=\varphi(s,u_0)(x).\\
\end{equation}
\end{definition}

Now, we have the following preliminary estimatives.

\begin{lemma}\label{le:1}
If $\omega$ is a synaptic kernel and $f,g$ are activation and modulation learning functions respectively, then the operator $J$ is such that 
\begin{equation}
\|Ju\|_{L^{\infty}_xL^{\infty}_t}\le (1 + \gamma)C_w.
\end{equation}
\end{lemma}

\begin{proof}
\begin{align}
|(Ju)(x,t)| &= \left |\int_{\mathbb{R}^m}{w(x,y)f(u(y,t)) \; dy} + \gamma \int_{\mathbb{R}^m}{w(x,y)g(u(x,t) - u(y,t)) f(u(y,t))}\; dy\right|\\
&\le C_w + \gamma\int_{\mathbb{R}^m}{|w(x,y)g(u(x,t) - u(y,t))| \; dy}\\
&\le C_w + \gamma\int_{\mathbb{R}^m}{|w(x,y)| \; dy}\\
&\le (1 + \gamma)C_w.
\end{align}
\end{proof}

\begin{lemma}[Global estimates]
\label{global}
Let $u_0 \in C_b(\RR^m)$ and $u$ be a solution to (\ref{eq:opkerdyn}). Then there is $C > 0$ such that
\begin{align}
	&|u(x,t)| \le C,& \mbox{for all } (x,t) \in \RR^m \times [0,\infty),
\end{align}
 with $C = \max\{\Vert u_0 \Vert_{\infty}, |(1+\gamma)C_w|\}$.
\end{lemma}
\begin{proof}
By the lemma \ref{le:1} we have that $-c \le (Ju)(x,t) \le c$ for some constant $c$. Then using \eqref{E:PVI1} we obtain the following differential inequalities:
\begin{align}
&u' \le -bu + c,\\
&u' \ge -bu - c,
\end{align}

with $b = 1$ and $ c= 1 + \gamma C_w$.\\

If we consider the IVP:
\begin{equation}
\label{eq:ode}
\left\{ 
  \begin{array}{l l}
    y' = -by + c,\\
   y(0) = a,
  \end{array} \right.
\end{equation}
the solution is given by:
\begin{align}
y(t) &= \frac{c}{b}\left(1 - e^{-bt}\right) + ae^{-bt}, \qquad t > 0.
\end{align}

If we plug it in  \eqref{eq:ode}, we get
\begin{equation}
y' = b\underbrace{\left(\frac{c}{b} - a\right)}_{p}e^{-bt}, \qquad y(0) = a.
\end{equation}

Note that if $p < 0$, $y$ will be monotone decreasing and, in that case, $c/b\le y \le a$. On the contrary, if $p>0$, $y$  will be monotone increasing and approaching $c/b$, then $a \le y \le c/b$. Therefore
\begin{equation}
y \le \max\{|a|, |c/b|\}.
\end{equation}
Following the same reasoning with
\begin{equation}
\label{eq:ode2}
\left\{ 
  \begin{array}{l l}
    y' = -by - c,\\
   y(0) = a,
  \end{array} \right.
\end{equation}
we get
\begin{equation}
y \ge -\max\{|a|, |c/b|\}.
\end{equation}
Then, changing $y$ by $u$, $a$ by $u_0$, $b$ and $c$ as above, we get the bound which we were looking for:
\begin{equation}
|u(x,t)| \le  \max\{\Vert u_0 \Vert_{\infty}, |(1+\gamma)C_w|\}.
\end{equation}
\end{proof}

Hereafter, the functions $\omega,f\,\text{and}\, g$ are as defined above.
\begin{lemma}
\label{lemmaF}
The operator $F$ is well defined as a map from $C^{0,1}_{b}(\mathbb{R}^m\times[0,\infty))$ on itself. Moreover, If $u(x,\cdot)$ is continous in t, then $(Ju)(x,\cdot)$ is continous in $t$ too.
\end{lemma}

\begin{proof}
Suppose $u \in C^{0,1}_{b}(\mathbb{R}^m\times[0,\infty))$. We need to look only at the second term of $Fu$, $(Ju)(x,t)$. We already saw in the preceding proof that $|(Ju)(x,t)|\le (1+\gamma)C_w$.\\

Let be $x,\tilde{x} \in \RR^m$. We have
\begin{equation}
	|(Ju)(x,t) - (Ju)(\tilde{x},t)|\le  I_1 + I_2 + I_3,
\end{equation}
where
\begin{equation}
I_1=\int_{\mathbb{R}^m}|w(x,y)-w(\tilde{x},y)||f(u(y,t))|dy\le k_w |x - \tilde{x}|,
\end{equation}
\begin{equation}
I_2=\gamma\int_{\mathbb{R}^m}\left\vert[w(x,y)-w(\tilde{x},y)]g(u(x,t)-u(y,t))\right\vert dy\le \gamma k_w |x - \tilde{x}|
\end{equation}
and
\begin{equation}
I_3=\gamma\int_{\mathbb{R}^m}|w(\tilde{x},y)[g(u(x,t)- u(y,t)) - g(u(\tilde{x},t)- u(y,t))]||f(u(y,t))|dy.
\end{equation}

It is clear that $I_1,I_2\to0$ if $x \rightarrow \tilde{x}$. $I_3\to0$ when $x \rightarrow \tilde{x}$ because of the Lebesgue's Dominated Convergence Theorem. So, we have continuity with respect to $x$.\\

Let's see if $Fu$ is continuously differentiable with respect to $t$. We have this for $u(x,t)$, so we check $Ju$.\\

As $f,g \in C_b^1(\mathbb{R})$ and $u(x,\cdot) \in C_b^1([0,\infty))$, we can apply the Leibniz rule and the chain rule to obtain:
\begin{align}
	\frac{\partial Ju}{\partial t}(x,t) &= \frac{\partial }{\partial t} \left(\int_{\mathbb{R}^m}{w(x,y)[1 + \gamma g(u(x,t) - u(y,t))] f(u(y,t))\;dy}\right)\\
	&=\int_{\mathbb{R}^m}w(x,y)\left\{ \frac{df}{ds}(u(y,t))\cdot\frac{\partial u}{\partial t}(y,t)\right\}dy +\notag\\
	&\mathrel{\phantom{=}} \gamma \int_{\mathbb{R}^m} w(x,y)\left[ g(u(x,t) - u(y,t))\left\{ \frac{df}{ds}(u(y,t))\cdot\frac{\partial u}{\partial t}(y,t)\right\} + \right.\notag\\
	&\mathrel{\phantom{=}}f(u(y,t))\left.\left\{ \frac{dg}{ds}(u(x,t) - u(y,t))\cdot\left(\frac{\partial u}{\partial t}(x,t)-\frac{\partial u}{\partial t}(y,t)\right)\right\}\right]dy.
\end{align} Because $df/ds$, $du/ds$ and $dg/ds$ are continuous and bounded, we get that $\frac{\partial Ju}{\partial t}(x,t)$  is continuous and $\sup_{x,t}|\frac{\partial Ju}{\partial t}(x,t)| < \infty$.
\\

Finally, if $u(x,\cdot)$ is continous in $t$ then $(Ju)(x,\cdot)$ is continous in $t$ by application of the dominated convergence theorem as in the previous steps.
\end{proof}

\begin{lemma}\label{lemavolterra}
 $u$ is a solution of \eqref{eq:opkerdyn} if and only if it is a solution of \eqref{eq:nfvolterra}. 
\end{lemma}
\begin{proof}
 If $u$ is a solution of \eqref{eq:nfvolterra}, we need to ensure that it is differentiable with respect to time. However, $Fu$ is continous, so $\int_0^t(Fu)(x,s)ds$ is differentiable with respect to $t$. Now we can differentiate \eqref{eq:nfvolterra} and $u$ satisfies \eqref{eq:opkerdyn}.\\
 
 If $u$ is a solution of \eqref{eq:opkerdyn}, we can integrate it to satisfy \eqref{eq:nfvolterra}.
\end{proof}

\begin{remark}
 The above lemmas are another versions of those obtained by Potthas and Grabem in the section 2 \textcolor{red}{[...]} in the case $\gamma=0$ (when the Amary's model works). The proofs follow those same ideas.
\end{remark}

 \section{Well posedness in $C_b(\mathbb{R}^n)$}
For $\rho>0$ given, we write $X_\rho:=C_b(\mathbb{R}^n\times[0,\rho])=C([0,\rho];C_b(\mathbb{R}^n))$. 
\begin{lemma}
\label{lemmawell}
The operator $A$ is well defined as an map on $X_\rho$. 
\end{lemma}
\begin{proof}
Let $u \in X_\rho$. That $(Au)(\cdot,t) \in C_b(\RR^m)$ is a direct consecuence of previous lemmas. We only need to see that $(Au)(x,\cdot)$ is continuous. Let us take $t_1, t_2 \in (0,\rho]$. Then for $x$ fixed,
\begin{align}
|(Au)(x,t_1) - (Au)(x,t_2)| &= \left|\int_0^{t_1} (Fu)(x,s) ds - \int_0^{t_2} (Fu)(x,s) ds\right|,\\
&= \left|\int_{t_2}^{t_1} (-u(x,s) + (Ju)(x,s)) ds\right|,\\
&\le \int_{t_2}^{t_1} \left|u(x,s)\right| ds + \int_{t_2}^{t_1} \left|(Ju)(x,s)) \right|ds,\\
&\le C|t_1 - t_2| + (1 + \gamma)C_w|t_1 - t_2|,
\end{align}
so, $(Au)(x,\cdot)$ is continuous on $(0,\rho]$. The continuity at zero can be seen by noting that
\begin{align}
|(Au)(x,t_1) - (Au)(x,0)| &=  \left|\int_{0}^{t_1} (-u(x,s) + (Ju)(x,s)) ds\right|.\\
&\le Ct_1 + (1 + \gamma)C_w(t_1).
\end{align} 
 This completes the proof.
\end{proof}
\begin{theorem}
\label{contraction}
The operator $A$ is a contraction in $X_\rho$ for $\gamma = \gamma_0$ fixed and $\rho > 0$ suffieciently small.
\end{theorem}
\begin{proof}
The operator $A$ can be divided in three different operators as 
\begin{equation}
A = A_1 + A_2 + A_3, 
\end{equation}
where $A_1$ is the linear integral operator
\begin{equation}
\label{eq:A1}
(A_1v)(x,t):=-\int_0^tv(x,s)\;ds,\quad 0\le t \le \rho,
\end{equation}
and $A_2$ and $A_3$, are the nonlinear integral operators
\begin{equation}
\label{eq:A2}
(A_2v)(x,t):=\int_0^t  \int_{\mathbb{R}^m} w(x,y)f(v(y,s)) \;dy\;ds,\quad 0\le t \le \rho,
\end{equation}
\begin{equation}
\label{eq:A3}
(A_3v)(x,t):=\gamma\int_0^t  \int_{\mathbb{R}^m} w(x,y)g(v(x,s) - v(y,s))f(v(y,s)) \;dy\;ds,\quad 0\le t \le \rho.
\end{equation}
Let $u := u_1 - u_2$. Then $A_1$ satisfies
\begin{equation}
\Vert A_1u \Vert_{\rho}\le \rho \left\Vert u  \right\Vert_\rho.
\end{equation}

Since $f \in C_b^1(\mathbb{R})$, then there is a constant $L$ such that $|f(s) - f(\tilde{s})| \le L|s - \tilde{s}|$ for all $s,\tilde{s} \in \mathbb{R}$. Therefore
\begin{align}
\int_{\mathbb{R}^m} |w(x,y)||f(u_1(y,t)) - f(u_2(y,t))| \;dy &\le LC_w\Vert u_1 - u_2\Vert_{\rho},
\end{align}
and 
\begin{align}
\vert (A_2u_1)(x,t) - (A_2u_2)(x,t)\vert &\le \int_0^t  \int_{\mathbb{R}^m}\vert w(x,y)\vert \vert f(u_1(y,s)) - f(u_2(y,s))\vert dy ds\\
&\le \rho LC_w\Vert u_1 - u_2 \Vert_{\rho},
\end{align}
so
\begin{equation}
\Vert A_2u_1 - A_2u_2 \Vert_{\rho} \le \rho LC_w\Vert u_1 - u_2 \Vert_{\rho}.
\end{equation}

Finally, we obtain an estimative for $A_3$.\\

Considering that function $g$ is also Lipschitz with Lipschitz constant $K$, we have

\begin{align}
&\left\vert\int_{\mathbb{R}^m} w(x,y) \big[ g(u_1(x,t) - u_1(y,t))f(u_1(y,t)) -  g(u_2(x,t) - u_2(y,t))f(u_2(y,t))\big]\;dy \right\vert\\
&\le \int_{\mathbb{R}^m}  \big|w(x,y) \big| \big|g(u_1(x,t) - u_1(y,t)) \big| \big|f(u_1(y,t)) - f(u_2(y,t)) \big|\; dy +\notag\\
&\mathrel{\phantom{=}} \int_{\mathbb{R}^m}  \big|w(x,y) \big| \big|f(u_2(y,t)) \big| \big|g(u_1(x,t) - u_1(y,t)) - g(u_2(x,t) - u_2(y,t)) \big| \;dy\\
&\le L\int_{\mathbb{R}^m} \big |w(x,y) \big| \big|u_1(y,t) - u_2(y,t) \big| \;dy + \notag\\
&\mathrel{\phantom{=}} K \int_{\mathbb{R}^m}  \big|w(x,y) \big| \big|u_1(x,t) - u_2(x,t) + u_2(y,t) - u_1(y,t) \big|\; dy\\
&\le L \Vert u_1 - u_2 \Vert_\infty \int_{\mathbb{R}^m}  \big|w(x,y) \big| \;dy + 2K \Vert u_1 - u_2 \Vert_\infty \int_{\mathbb{R}^m}  \big|w(x,y) \big|\;dy\\
&\le C_w (L + 2K) \Vert u_1 - u_2 \Vert_\rho,
\end{align}
then, 
\begin{align}
|(A_3u_1 - A_3u_2 )(x,t)| &\le \gamma \int_0^t  \int_{\mathbb{R}^m} \big\vert w(x,y)[g(u_1(x,s) - u_1(y,s))f(u_1(y,s)) - \notag \\ 
&\mathrel{\phantom{=}}  \qquad \qquad g(u_2(x,s) - u_2(y,s))f(u_2(y,s))] \big\vert dy ds \\
&\le \gamma\rho C_w (L + 2K) \Vert u_1 - u_2 \Vert_\rho,
\end{align}
and therefore,
\begin{equation}
\Vert A_3u_1 - A_3u_2 \Vert_{\rho} \le \gamma \rho (L+ 2K)C_w\Vert u_1 - u_2 \Vert_{\rho}.
\end{equation}
Summarizing, we have
\begin{equation}
\label{eq:condicion}
\Vert Au_1 - Au_2 \Vert_\rho \le \rho[1 + LC_w + \gamma(L+2K)C_w]\Vert u_1 - u_2\Vert_\rho,
\end{equation}
thus, $A = A_1 + A_2 + A_3$ is a contraction for $0 < q :=  \rho[1 + LC_w + \gamma(L+2K)C_w] < 1$ .
\end{proof}
\begin{theorem}[Local existence of solutions]
\label{thmexistence}
Let $L$ and $K$ being the Lipschitz constants of $f$ and $g$ respectively, $\gamma = \gamma_0$ fixed and $\rho > 0$ such that $0 < q < 1$. Then for all $u_0 \in C_b(\RR^m)$ there exists a unique solution $u \in X_\rho$ for the IVP \eqref{eq:opkerdyn}.
\end{theorem}
\begin{proof}
The space $X_\rho$ is a Banach space and the operator $\tilde{A}u = u_0 + Au$ is a contraction on $X_\rho$. Then the equation $u = \tilde{A}u$ has an unique fixed point $u^*$, which corresponds with the solution to \eqref{eq:nfvolterra} and, by lemma \ref{lemavolterra}, to the solution of \eqref{eq:opkerdyn}. Because the operator $\tilde{A}$ is defined on all the space $X_\rho$, the solution to  \eqref{eq:opkerdyn} is unique.
\end{proof}

\begin{theorem}[Global existence of solutions]
\label{globalc}
For all $u_0 \in C_b(\RR^m)$ there is a solution $u \in C_b(\RR^m \times [0,\infty))$.
\end{theorem}
\begin{proof}
Note that neither \eqref{eq:opkerdyn} nor the obtained estimates depend on time explicitly and $\rho$ in theorem \ref{thmexistence} does not depend on the initial data $u_0$. Also, the theorem \ref{global} guarantees that the solution is bounded. Then, we can repeat the presented argument for initial data $u(x,\rho) \in C_b(\RR^m)$ and extend the solution to an interval $[\rho, 2\rho)$. Iterating this, we can cover the entire interval $[0,\infty)$.
\end{proof}
\begin{theorem}[Continuous dependence on initial data]
\label{dependence1}
Let $u$ and $v$ be solutions to \eqref{eq:nfvolterra}, with initial data $u_0$ and $v_0$ respectively. Then, for a fixed $\gamma$ and for every $\rho>0$, there exists a constant $C = C(\rho)$ such that
\begin{equation}
\Vert u - v\Vert_\rho \le C\Vert u_0 - v_0 \Vert_{C_b(\mathbb{R}^m)}.
\end{equation}
\end{theorem}
\begin{proof}
From theorem \ref{contraction} we have
\begin{align}
\Vert Au - Av\Vert_\rho \le q\Vert u - v\Vert_\rho,
\end{align}
choosing $\rho>0$ such that $q: = \rho [1 + LC_w + \gamma(L + 2K)C_w] < 1$.\\

Then if $u$ and $v$ are solutions to \eqref{eq:nfvolterra}, with initial data $u_0$ and $v_0$ respectively, we obtain
\begin{align}
\Vert u - v\Vert_\rho \le \Vert u_0 - v_0\Vert_{C_b(\mathbb{R}^m)} + \Vert Au - Av\Vert_\rho \le \Vert u_0 - v_0\Vert_{C_b(\mathbb{R}^m)} + q\Vert u - v\Vert_\rho,
\end{align}
so
\begin{align}
\Vert u - v\Vert_\rho \le C\Vert u_0 - v_0\Vert_{C_b(\mathbb{R}^m)}
\end{align}
with $C = \frac{1}{1 - q}$.\\

Chosen that $\rho$, we can repeat the latter argument in $[\rho,2\rho]$ and we will obtain the same estimative. Repeating the argument many times as necessary, cover any interval. 
\end{proof}

The following result is an immediate consequence of theorem \ref{global}, and says that a purely excitatory network, remains excitatory

\begin{theorem}[Positivity]
 Let $u_0 \in C_b(\RR^m)$, and $\gamma > 0$. Suppose $u_0(x) \ge 0$ for all $x \in \RR^m$. If the synaptic kernel $w:\RR^m \times \RR^m \rightarrow \RR$, satisfy:	
\begin{equation}
 w(x,y) > 0 \quad \forall\, x,y \in \RR^m,
\end{equation}
then 
\begin{equation}
u(x,t) \ge 0\quad\forall\, x \in\RR^m, t \ge 0.
\end{equation}
 \end{theorem}
 
 \begin{proof}
We already saw in theorem \ref{global} that $c\le u \le a$ or $a \le u \le c$, with $c > 0$ and $a = u_0$. Therefore, we can conclude that $u(x,t) \ge 0$ for all $t \in [0,\infty)$. 
 \end{proof}

 Now, we will see that the Amary's model is a good approximation of the model proposed by Abbassian et al. More precisely:

\begin{theorem}[Annulling Plasticity Limit]
Let $u^\gamma$ and $u$ the solutions to the problems \eqref{E:PVI1}-\eqref{E:PVI2} and \eqref{E:Amari} with initial data $u_0^\gamma$ and $u_0$ respectively. If
\begin{equation}
\lim_{\gamma \rightarrow 0} \|u_0^\gamma - u_0 \|_{C_b(\mathbb{R}^m)} = 0,
\end{equation}
then:
 \begin{equation}
\lim_{\gamma \rightarrow 0} \| u^\gamma - u \|_{C_b(\mathbb{R}^m \times [0,\infty))} = 0.
\end{equation} 
\end{theorem}
\begin{proof}
We can put the IVP \eqref{E:Amari} like a Volterra integro differential equation too, similar to \eqref{eq:nfvolterra}, but with $\omega$ instead of $\omega_{g,\gamma}$. Then we have the contraction operators $A$ and $A^\gamma$, corresponding with the Amari's and Abbassian's models, respectively,\\

Let $\rho > 0$ be fixed and $0\le t<\rho$. We have the following estimatives
\begin{align}
N_1:= \displaystyle\int_{\mathbb{R}^m}\left|w(x,y)\right|\left|f(u(y,t)) - f(u^\gamma(y,t))\right| \; dy &\le LC_w\Vert u - u^\gamma\Vert_{C_b(\mathbb{R}^m \times [0,\infty))}\\
N_2:= \displaystyle\int_{\mathbb{R}^m}\left|w(x,y)\right|\left|g(u^\gamma(x,t) - u^\gamma(y,t))\right| \left|f(u^\gamma(y,t))\right|\; dy &\le KC_w\Vert u - u^\gamma\Vert_{C_b(\mathbb{R}^m \times [0,\infty))}.
\end{align}

Then 
\begin{align}
\left|Au(x,t) - A^\gamma u^\gamma (x,t)\right| &\le \int_0^t  \left( \left| u(x,s) - u^\gamma(x,s) \right| + N_1 + \gamma N_2 \right) \;ds\\
&\le \rho[1 + (L + \gamma K)C_w]  \Vert u - u^\gamma\Vert_{C_b(\mathbb{R}^m \times [0,\infty))},
\end{align}
so,
\begin{align}
\label{eq:ineq}
	| u(x,t) - u^\gamma (x,t) | &\le  | u_0(x) - u_0^\gamma(x) |  + \left|Au(x,t) - A^\gamma u^\gamma (x,t)\right|\\
	&\le \Vert u_0 - u_0^\gamma\Vert_{C_b(\mathbb{R}^m)} + \rho[1 + (L + \gamma K)C_w]  \Vert u - u^\gamma\Vert_{C_b(\mathbb{R}^m \times [0,\infty))}.
\end{align}

Therefore, if $\rho$ is chosen such that $\rho[1+(L+\gamma K)C_w]<1$, we have
\begin{align}
	\Vert u - u^\gamma \Vert_\rho &\le  \frac{1}{1 - \rho[1 + (L + \gamma K)C_w]}\Vert u_0 - u_0^\gamma \Vert_{C_b(\RR^m)}.
\end{align}

Then $u^\gamma\to u$ uniformily on $\mathbb{R}^n\times[0,\rho]$ when $\gamma\to0$. Because the estimatives do not depend explicitly of time, we can iterate this procedure in $[\rho,2\rho)$ and so on, with the same $\rho$, to cover the entire interval $[0,\infty)$.

\end{proof}

\section{Well posedness in $L^1(\Omega)$}
\label{wpl}
The brain is not a continuous medium and it receives and emits no continuous signals, then is interesting to look for solutions for non continous data in a bounded spatial domain.\\

So, in this section we are going to consider the model of Abbassian et al
\begin{equation}
\label{eq:kerdynb}
\left\{ 
  \begin{array}{l l}
    u_t(x,t) = -u(x,t) + \displaystyle\int_{\Omega}{\omega_{g,\gamma}f(u(y,t))\; dy} ,&(x,t) \in \Omega\times (0,\infty)\\
     u (x,0) = u_0(x),&x \in \Omega,
      \end{array} \right. 
\end{equation}
with $u_0 \in L^1(\Omega)$ and $\Omega \subset \RR^m$ bounded.\\

Let $\rho>0$ fixed. We consider the space of continous functions with respect to $t$, defined on $[0,\rho]$ with values in $L^1(\Omega)$ and we denote that space by 
\begin{equation}
 X_{\rho}:=C([0,\rho];L^1(\Omega)).
\end{equation}

Now, a solution of \eqref{eq:kerdynb} will be a function $u\in X_{\rho}$ such that it satisfies \eqref{eq:kerdynb}.\\

The following results ensures the well posedness of the Cauchy problem \eqref{eq:kerdynb}.

\begin{lemma}
If $u$ is solution of the IVP (\ref{eq:kerdynb}), then it is bounded.
\end{lemma}
\begin{proof}
For a solution $u$ of (\ref{eq:kerdynb}) we have that
\begin{equation}
\label{variation}
u = e^{-t}u_0 + \int_0^t e^{s -t} (Ju)(x,s) ds,
\end{equation}
then
\begin{align}
\Vert u(\cdot,t)\Vert_{L^1(\Omega)} &\le  e^{-t}\Vert u_0(x)\Vert_{L^1(\Omega)} + \int_0^t e^{s-t}\left \Vert Ju\right\Vert_{L^1(\Omega)} \; ds,\\
&\le  e^{-t}\Vert u_0(x)\Vert_{L^1(\Omega)} + C_w |\Omega| (1 - e^-t),\\
&\le  \Vert u_0(x)\Vert_{L^1(\Omega)} + C_w |\Omega|,
\end{align}
so,
\begin{equation}
\sup_t  \Vert u(\cdot,t)\Vert_{L^1(\Omega)} \le \infty.
\end{equation}
\end{proof}
\begin{theorem}
Let $u_0 \in L^1(\Omega)$. The operator $A:X_\rho \rightarrow X_\rho $, as defined in \eqref{eq:op3}, is well defined.
\end{theorem}
\begin{proof}
Let $0\le t < \rho$, $u \in X_\rho$. Then, we have the following estimate for the operator $J$:
\begin{align}
\Vert (Ju)(\cdot, t) \Vert_{L^1(\Omega)} &= \int_{\Omega}\left\vert \int_{\Omega}w(x,y)[1 + \gamma g(u(x,t) - u(y,t))]f(u(y,t)) \;dy \right\vert \;dx\\
&\le \int_{\Omega}\int_{\Omega}\left\vert w(x,y)[1 + \gamma g(u(x,t) - u(y,t))]f(u(y,t))  \right\vert \;dy \;dx\\
&\le \int_{\Omega}\int_{\Omega}\left\vert w(x,y) \right\vert \;dy \;dx\\
&\le C_w |\Omega|,
\end{align}
where $|\Omega|$ is the measure of $\Omega$.\\

Since $ u(\cdot,t) \in L^1(\Omega)$ we have $(Fu)(\cdot,t) \in L^1(\Omega)$, so $(Au)(\cdot,t) \in L^1(\Omega)$ by the application of Fubini's theorem. The $(Au)(x,\cdot)$ can be shown to be continuous on $[0,\rho]$  in an analogous manner as was done in the lemma  \ref{lemmawell}.
\end{proof}
\begin{theorem}
\label{thmbound}
Let $u, v \in X_\rho$ are solutions to \eqref{eq:kerdynb}. Then there exists a constant $C = C(\Omega,w,g,f, \gamma)$ such that
\begin{equation}
\Vert Au - Av\Vert_\rho \le \rho C\Vert u - v \Vert_\rho
\end{equation}
\end{theorem}
\begin{proof}
Let us divide the operator $A$ in three operators $A_1$, $A_2$ and $A_3$, as previously (eqs \eqref{eq:A1}, \eqref{eq:A2}, \eqref{eq:A3}), such that: $A = A_1 + A_2 + A_3$.\\

We have the following estimate for the linear operator $A_1$
\begin{align}
\Vert (A_1u)(\cdot,t) - (A_1v)(\cdot,t) \Vert_{L^1(\Omega)} & \le \rho \Vert u - v \Vert_\rho.
\end{align}
For $A_2$, we obtain the following estimate
\begin{align}
\Vert (A_2u)(\cdot,t) - (A_2v)(\cdot,t) \Vert_{L^1(\Omega)} &\le  \int_0^t \int_{\Omega}\int_{\Omega} \left|w(x,y)\right| \left|f(u(y,s)) - f(v(y,s))\right| \; dy \; dx \;ds,\\
&\le  L  \int_0^t \int_{\Omega} \left|u(y,s) - v(y,s)\right| \int_{\Omega} \left|w(x,y)\right| \; dx \; dy \; ds,\\
&\le \rho L C_w |\Omega| \left\|u - v\right\|_\rho.
\end{align}
Finally, we can obtain an estimate for $A_3$:
\begin{align}
\Vert (A_3u)(\cdot,t) - (A_3v)(\cdot,t) \Vert_{L^1(\Omega)} &\le \gamma \int_{\Omega}  \int_0^t \int_{\Omega} \left|w(x,y)\right|\left|g(u(x,s) - u(y,s))f(u(y,s)) - \right.\notag \\
&\left. \mathrel{\phantom{=}}  g(v(x,s) - v(y,s))f(v(y,s))\right.|\; dy \; ds  \; dx\\
&= \gamma \int_0^t \int_{\Omega}  \int_{\Omega} \left|w(x,y)\right|\left|g(u(x,s) - u(y,s))f(u(y,s)) - \right.\notag \\
&\left. \mathrel{\phantom{=}}  g(v(x,s) - v(y,s))f(v(y,s))\right.|\; dy  \; dx \;ds\\
&= \gamma \int_0^t \int_{\Omega} \int_{\Omega} \left|w(x,y)\right|\left|g(u(x,s) - u(y,s))f(u(y,s)) - \right.\notag\\
&\left. \mathrel{\phantom{=}} g(v(x,s) - v(y,s))f(u(y,s)) + g(v(x,s) - v(y,s))f(u(y,s))-\right. \notag\\
&\left. \mathrel{\phantom{=}}  g(v(x,s) - v(y,s))f(v(y,s))\right.|\; dy  \; dx \; ds\\
&\le \gamma \left( \int_0^t \int_{\Omega} \int_{\Omega} \left|w(x,y)\right|\big|g(u(x,s) - u(y,s))- \right.\notag\\
& \left.\mathrel{\phantom{=}}g(v(x,s) - v(y,s))\big|\; dy  \; dx \; ds +  \right.\notag\\
& \left.\mathrel{\phantom{=}} \int_0^t\int_{\Omega} \int_{\Omega} \left|w(x,y)\right|\big|f(u(y,s)) - f(v(y,s))\big|\; dy  \; dx \;ds \right)\\
&\le \gamma \left(K \int_0^t \int_{\Omega} \int_{\Omega} \left|w(x,y)\right|\big|u(x,t) - v(x,t)\big|\; dy  \; dx \;ds +  \right.\notag\\
& \left.\mathrel{\phantom{=}} (L + K)\int_0^t \int_{\Omega} \int_{\Omega} \left|w(x,y)\right|\big|u(y,t) - v(y,t)\big|\; dy  \; dx \;ds \right)\\
&\le \rho\gamma K C_w |\Omega| \Vert u - v \Vert_\rho +  \rho\gamma(L+K) C_w |\Omega|\Vert u - v \Vert_\rho\\
&= \rho\gamma (L + 2K) C_w |\Omega| \Vert u - v \Vert_\rho.
\end{align}
Therefore
\begin{align}
\Vert Au - Av \Vert_{\rho} &\le \rho \Vert u - v \Vert_{\rho} +  \rho L C_w |\Omega| \left\|u - v\right\|_{\rho} + \rho\gamma (L + 2K) C_w |\Omega| \Vert u - v \Vert_{\rho}\\
&= \rho[1 + C_w|\Omega|(L + \gamma (L + 2K) )]\Vert u - v \Vert_{\rho}\\
&=\rho C\Vert u - v \Vert_{\rho},
\end{align}
with $C:=1 + C_w|\Omega|(L + \gamma (L + 2K) )$ independent of the time.
\end{proof}

\begin{theorem}[Existence and uniqueness in $L^1(\Omega)$]
For all $u_0 \in L^1(\Omega)$, there exists a unique solution $u \in C[[0,\infty);L^1(\Omega)]$ for the IVP \eqref{eq:kerdynb}.
\end{theorem}
\begin{proof}
Choose $\rho$ as in theorem \ref{thmbound} such that
\begin{equation}
C := \rho[1 + C_w|\Omega|(L + \gamma (L + 2K) )] < 1.
\end{equation}
Then the operator $A$ is a contraction in the Banach space $X_\rho$. The existence and uniqueness follows from the Banach Fixed Point Theorem applyied to the operator $\tilde{A}:X_\rho \rightarrow X_\rho$ given by $\tilde{A} = u_0 + A$, i.e.
\begin{equation}
(\tilde{A}u)(x,t) = u_0(x) + (Au)(x,t),
\end{equation}
which is also a contraction due to $A$. Now, this local result can be extended to the whole interva $[0,\infty)$ noting that the estimatives just obtained are the same and independent of initial data. For initial data $u(x, \rho) \in L^1(\Omega)$, a similar result can be obtained for $[\rho, 2\rho]$, and so on.
\end{proof}

\begin{theorem}[continuous dependence on initial data]
Let $u$ and $v$ be solutions to \ref{eq:kerdynb}, with initial data $u_0$ and $v_0$ respectively. Then, for a fixed $\gamma$, there exists $\rho > 0$ and a constant $C = C(\rho)$ such that
\begin{equation}
\Vert u - v\Vert_\rho \le C\Vert u_0 - v_0 \Vert_{L^1(\Omega)}
\end{equation}
\end{theorem}
\begin{proof}
This is consequence of the the operator being a contraction, similar as in theorem \ref{dependence1}.
\end{proof}

\section{Stationary solutions}

The following results guaranties the existence of stationary solutions in a general way, improving results obtained in \cite{abbassian2012neural}
\begin{lemma}\label{L:Stationary}
 Suppose that $u^\gamma=u^\gamma(x,t)$ is a global solution of \eqref{E:PVI1} in $X_\rho=C([0,+\infty);C_b(\Omega))$. Then the family $\{u^\gamma(\cdot,t)\}_{t\ge0}$ is equicontinous in $C_b(\Omega)$ if $\gamma$ is small enough and $\Omega$ is compact. In this case, we can choose a sequence of times $t_1<t_2<\dots,\; t_n\to +\infty$; such that $\{u^\gamma(\cdot,t_n)\}_{n\in\mathbb{N}}$ converges uniformily to a function $u_{\infty}=u_{\infty}(x)$ in $C(\Omega)$.
\end{lemma}
\begin{proof}
A global solution of \eqref{E:PVI1} is
\begin{equation}
u(x,t)=e^{-t}u_0(x)+\int_0^t e^{s -t}(Ju)(x,s)ds, 
\end{equation}
so
\begin{align}
|u(x,t)-u(\tilde{x},t)| &\le|u_0(x)-u_0(\tilde{x})|+\int_0^t e^{s-t}\left |(Ju)(x,s)-(Ju)(\tilde{x},s)\right|ds,\\
&\le|u_0(x)-u_0(\tilde{x})|+\Vert (Ju)(x,\cdot)-(Ju)(\tilde{x},\cdot) \Vert_{\infty}.
\end{align}
But
\begin{equation}
|(Ju)(x,t) - (Ju)(\tilde{x},t)|\le  I_1 + I_2 + I_3,
\end{equation}
whit
\begin{equation}
I_1\le k_w |x - \tilde{x}|,
\end{equation}
\begin{equation}
I_2\le \gamma k_w |x - \tilde{x}|
\end{equation}
and
\begin{equation}
I_3=\gamma\int_{\mathbb{R}^m}|w(\tilde{x},y)[g(u(x,t)- u(y,t)) - g(u(\tilde{x},t)- u(y,t))]||f(u(y,t))|dy.
\end{equation}

As $g(\delta)=e^{-\delta^2}$ es a contraction, then
\begin{equation}
I_3\le \gamma C_w |u(x,t)-u(\tilde{x},t)|,
\end{equation}
therefore
\begin{equation}
 \Vert u(x,\cdot)-u(\tilde{x},\cdot)\Vert_{\infty}\le |u_0(x)-u_0(\tilde{x})|+k_w |x - \tilde{x}|+\gamma k_w |x - \tilde{x}|+\gamma C_w \Vert u(x,\cdot)-u(\tilde{x},\cdot)\Vert_{\infty}.
\end{equation}

If $\Omega$ is compact then $u_0$ is uniformly continuous, so we can choose $\gamma$ small enough such that $\{u^\gamma(\cdot,t)\}_{t\ge0}$ is equicontinous. Since this family is bounded, the Arzela theorem guarantees a sequence $\{u^\gamma(\cdot,t_n)\}_{n\in\mathbb{N}}$ convergent uniformily to a function $u_{\infty}=u_{\infty}(x)$ in $C(\Omega)$.
\end{proof}

\begin{lemma}
Suppose that $\varphi_n\to \varphi$ uniformily in $S$. Let $x$ a limit point of $S$ and suppose that $\lim_{t\to x}\varphi_n(t)=a_n$. Then $\{a_n\}_n$ is convergent and $\lim_{n\to\infty}a_n=\lim_{t\to x}\varphi(t)$, that is
\begin{equation}
 \lim_{n\to\infty}\lim_{t\to x}\varphi_n(t)=\lim_{t\to x}\lim_{n\to\infty}\varphi_n(t).
\end{equation}
\end{lemma}

\begin{lemma}
Suppose that $\varphi_n\to \varphi$ uniformily in $S$ and $\{\varphi_n\}_{n\in\mathbb{N}}$ is uniformily bounded, that is, there is a $M>0$ such that $|\varphi_n(x)|\le M\;\text{for all x}\;\text{in}\; S,n\in\mathbb{N}$. If $\psi$ is a continuous function on the closed the ball $\bar{B}(0;M)$, then $(\psi\circ \varphi_n)\to (\psi\circ \varphi)$ uniformily in $S.$  
\end{lemma}

The above lemmas are well known results from advanced calculus.\\

Finally, we have
\begin{theorem}
\label{asymptotic}
If $\Omega$ is compact, then for sufficiently small values of $\gamma\ge0$ the equation \eqref{E:PVI1} has a stationary solution in $\Omega$.
\end{theorem}

\begin{proof}
Let $\{u^\gamma(\cdot,t_n)\}_{n\in\mathbb{N}}$ with limit $u_\infty$, as in lemma \ref{L:Stationary}. We have that $f(u^\gamma(\cdot,t_n))\to f(u_\infty(\cdot))$ uniformily in $\Omega$ because the previous lemma. Then 
\begin{equation}
\int_{\Omega}\omega_{g,\gamma}(\cdot,y)f(u^\gamma(y,t_n))dy\to\int_{\Omega}\omega_{g,\gamma}(\cdot,y)f(u_\infty(y))dy 
\end{equation}
uniformily in $\Omega$ because of $\int_\Omega|\omega(x,y)|dy\le C_\infty|\Omega|.$\\

As
\begin{equation}
 \partial_tu^\gamma(\cdot,t)=\lim_{h\to0}\frac{u^\gamma(\cdot,t+h)-u^\gamma(\cdot,t)}{h}
\end{equation}
exists, and
\begin{equation}
\lim_{n\to\infty}\delta_{t_n}(h)=0 \quad\text{(uniformily)}
\end{equation}
with
\begin{equation}
\delta_{t_n}(h):=\frac{u^\gamma(\cdot,t_n+h)-u^\gamma(\cdot,t_n)}{h}, 
\end{equation}
then we can take limit in equation \eqref{E:PVI1} and get
\begin{equation}
u_\infty(x) =\displaystyle\int_{\Omega}\omega_{g,\gamma}(x,y)f(u_\infty(y))dy,\quad\forall x\in\Omega.
\end{equation}
Therefore, $u_\infty$ is a stationary solution of the equation \eqref{E:PVI1}.
\end{proof}

\subsection{The result of learning: the gain field equation}\label{gainfield}

Now we develop a novel connection between learning kernels and gain field equations. In the stationary case, the kernel $\omega_{g,\gamma}$ reaches its final form 
\begin{equation}
\bar{\omega}_{g,\gamma}(x,y) = \omega (x,y)G(u_\infty(x), u_\infty(y)) = \omega (x,y) [1 - g( u_\infty(x) - u_\infty(y))].
\end{equation}
This kernel represents the distribution of synaptic weights \emph{after} the learning takes place. We can now plug this kernel $\bar{\omega}_{g,\gamma}$ into a new equation
\begin{equation}
u_t(x,t)+u(x,t)=\int_{\Omega}\omega(x,y)G(u_\infty(x), u_\infty(y))f(u(y,t))dy,
\label{eq:staticplasticity}
\end{equation}
and study the dynamics of the Amari neural field under the learned spatial patterns. Note that in this equation, the kernel does not change anymore in an activity dependent manner. The domain $\Omega$ can be bounded or unbounded. The learned kernel can be divided in \emph{pre-synaptic} and \emph{post-synaptic} contributions by using Mercer's theorem \cite{riesz1955b,sun2005mercer}, which states that positive semidefinite kernels on $\Omega$ can be decomposed as
\begin{equation}
\label{eq:mercer}
G(x,y) = \sum_i \sigma_i \phi_i(x) \phi_i(y)
\end{equation}
where $\phi_i$ and $\sigma_i$ are sequences of orthonormal eigenfunctions and the corresponding eigenvalues , $i = 1,\ldots,\infty$. This decomposition can be nicely interpreted as the presinaptic and postsynaptic effects of learning. It is known that, presynaptically, learning increases the probability of the neurotransmitter release, while postsynaptically it affects the density of receptors\cite{abbott2000synaptic}. Presynaptic effects can be translated in terms of an effectiv gain field. For that we have then the following estimate
\begin{align}
 \sum_i \sigma_i \phi_i(x) \phi_i(y) &\le  \left(\sum_i |\sigma_i \phi_i(x) \phi_i(y)|\right)^2\nonumber\\
&\le  \sum_i \sigma_i |\phi_i(x)|^2  \sum_i \sigma_i|\phi_i(y)|^2\nonumber\\
&= \left[\max_{x \in \Omega} G(x,x)\right]\sum_i \sigma_i|\phi_i(y)|^2.
\end{align}
The last term in the inequality 
\begin{equation}
\label{eq:pre}
\phi_{\mathrm{pre}}(y) = K_{\mathrm{pre}}\sum_i \sigma_i|\phi_i(y)|^2,
\end{equation}
is called \emph{pre-synaptic gain field}, for some constant $K_{\mathrm{pre}} > 0$. The pre-synaptic gain field equation is then defined as
\begin{equation}
u_t(x,t)+u(x,t)= \int_{\Omega}\omega(x,y) \phi_{\mathrm{pre}}(y)f(u(y,t))dy.
\label{eq:gainfield}
\end{equation}
It is easy to show that the particular condtions on the learning kernel (definition \ref{def:learning}) imply that it is positive semidefinite, that is $\int \int \phi(x) G(x,y) \phi (y) dx dy \ge 0$ for all $\phi \in L^2(\Omega)$. We state the following lemma without proof
\begin{lemma}
Let $u^\gamma$ and $u^p$ be solutions of  \eqref{E:Abbassian} and \eqref{eq:gainfield}, then, $u^p(x,t) >u^\gamma(x,t)$  for all $t > T$ for some $T > 0$.
\end{lemma}
We now concentrate in a particular form of \eqref{eq:pre} and synaptic kernels that allows us to introduce a new methodology to analyze gain fields emerging from learning processes. For $K_{\mathrm{pre}} = 1/\lambda$,
\begin{equation}
\phi_{\mathrm{pre}}(x)  =  \frac{1}{\lambda} (k^2 - V(x)),
\label{eq:staticplasticity3}
\end{equation}
This function represents a spatial gain distribution relative to a background or base gain $k^2$. The synaptic kernel takes the form of an exponentially decaying function and it is assumed spatially homogeneous, possitive (effectively excitatory) and isotropic
\begin{equation}
w(x,y) = w(|x-y|) =\frac{1}{2}e^{-\lambda|x-y|}, \;\; \lambda >0.\label{eq:weight}
\end{equation}
For linear firing rate mappings $f(u(x)) = u(x)$ and under the previous assumptions, equation \eqref{E:Amari} becomes
\begin{equation}
u_t(x,t)=-u(x,t)+\frac{1}{2\lambda}\int_{\mathbb{R}^m}e^{-\lambda|x-y|}\phi_{\mathrm{pre}}(y)u(y,t)dy,
\label{eq:amarigain}
\end{equation}

This model has been studied before in the literature for the case that we call the \textit{free network} $V=0,$ applied for problems of non-local diffusion and in population dynamics. This model can also be related to the time independent Schr\"odinger equation, a relation which allows us to stydy presynaptic gain fields in analogy to \emph{quantum wells}. First, suppose that the solutions to equation (\ref{eq:amarigain}) approach an stationary state $u(x)$, given by
\begin{equation}
	u(x)=\frac{1}{2\lambda}\int_{\mathbb{R}}e^{-\lambda|x-y|}P(y)u(y)dy,
	\label{eq:amaristeady}
\end{equation}
which can be writen as
\begin{equation}
u(x)=w\ast[{f\circ{u(\cdot,t)}}](x), \label{eq:conv}
\end{equation}
where $\ast$ denotes spatial convolution and $\circ$ the function composition. Then, we can apply the Fourier Transform (FT) with respect to the spatial variable to obtain
\begin{equation}
\hat{u}=\widehat{w \ast f\circ{u}},\label{eq:ftamari}
\end{equation}
and using the fact that $\widehat{w\ast f} = \hat{w}\hat{f}$ we get
\begin{equation*}
\hat{u}(\xi)=\frac{1}{\xi^{2}+\lambda^{2}}\widehat{f\circ{u}}(\xi).
\end{equation*}
Multiplying by $\hat{w}$ on both sides and taking the inverse FT and using (\ref{eq:staticplasticity3}) we can write
\begin{equation}
\lambda^{2}{u}(x)-u_{xx}(x)= [k^2 - V(X)]u(x),\label{e10}
\end{equation}
that is
\begin{equation} 
-\frac{d ^2 u}{dx^2}(x)+V(x)u(x)=Eu(x),\label{eq:schrodinger}
\end{equation}
with $E = k^2 - \lambda^2$, which is the time-independent Schr\"{o}dinger equation.\\

We believe this relationship allow us to frame three questions as inquiries about the eigenfunction $u(x)$:
\begin{enumerate}
\item Under which conditions can a focus of asynchronous activity be propagated from a region of high gain to a region of low gain?
\item Under which conditions of the gain field can regions of high (up states) and low (down states) coexist in the same spatial domain?
\item Under which conditions can this network support bumps?
\end{enumerate}
We answer those questions in \cite{Jimenez2017}. In particular, we show that stable solutions are posible in the presence of gain fields even if the neural field has completely excitatory synaptic kernels and unbounded firing rate mappings.\\
\subsection*{Remark}
Note that if $V$ is a square well potential with height $k^2$, then, the integral in  \eqref{eq:amaristeady}is over a compact interval, therefore,  therefore, the existence of $u$ is well justified.\\

\section{Conclusions}
\label{conclusions}
We have shown that the model proposed by Abbassian et al. is well defined for general choices of the function $f$, $w$ and $g$. Moreover, our results guarantee the existence of solutions in general functional spaces and the existence of the stationary state, which put the conclusions reached in that paper in a solid analytic ground for further development. \\
In the case in which $f$ is a heaviside function, the IVP turns out to be ill posed [Cordero y Pinilla (to be published)]. This is because the equation's flow is discontinous and, therefore, a fixed point argument can not be used in $X_\rho$. In that case, one expects to proof the existence of weaker solutions by using compactness arguments as in \cite{potthast2010existence}. Other aspects like the speed, and decay of traveling waves can also be studied in this, more general, case. \\
The final form of the kernel is important in that it represent the final structure of the synaptic weight after learning. We have shown that this structure can be decomposed in pre-synaptic and post-synaptic distributions, of which we have focused in the pre-synaptic ones. We derive a form of the time independent Schr\"odinger equation that can be used to study particular distributions of gains. This methodology is further developed in future work.

\section{Acknowledgements}
\label{acknowledgements}
This project, with title  ``An\'alisis te\'orico - Experimental de un modelo de campo Neural usando t\'ecnicas de la mec\'anica cu\'antica", was funded by the ``Convocatoria del programa nacional de proyectos para el fortalecimiento de la investigaci\'on, la creaci\'on y la innovaci\'on de posgrados de la universidad Nacional de Colombia 2013 - 2015", code number 19375, of the Universidad Nacional de Colombia, sede Manizales.




\bibliographystyle{elsarticle-num} 
\bibliography{article}

\begin{thebibliography}{10}
\expandafter\ifx\csname url\endcsname\relax
  \def\url#1{\texttt{#1}}\fi
\expandafter\ifx\csname urlprefix\endcsname\relax\def\urlprefix{URL }\fi
\expandafter\ifx\csname href\endcsname\relax
  \def\href#1#2{#2} \def\path#1{#1}\fi

\bibitem{coombes2014neural}
S.~Coombes, P.~B. Graben, R.~Potthast, J.~Wright, Neural Fields: Theory and
  Applications, Springer, 2014.

\bibitem{wilson1972excitatory}
H.~R. Wilson, J.~D. Cowan, Excitatory and inhibitory interactions in localized
  populations of model neurons, Biophysical journal 12~(1) (1972) 1.

\bibitem{coombes2005bumps}
S.~Coombes, M.~Owen, Bumps, breathers, and waves in a neural network with spike
  frequency adaptation, Physical Review Letters 94~(14) (2005) 148102.

\bibitem{kao2016absolute}
C.-Y. Kao, C.-W. Shih, C.-H. Wu, Absolute stability and synchronization in
  neural field models with transmission delays, Physica D: Nonlinear Phenomena
  328 (2016) 21--33.

\bibitem{abbott2000synaptic}
L.~F. Abbott, S.~B. Nelson, Synaptic plasticity: taming the beast, Nature
  neuroscience 3 (2000) 1178--1183.

\bibitem{abbassian2012neural}
A.~Abbassian, M.~Fotouhi, M.~Heidari, Neural fields with fast learning dynamic
  kernel, Biological cybernetics 106~(1) (2012) 15--26.

\bibitem{fotouhi2015continuous}
M.~Fotouhi, M.~Heidari, M.~Sharifitabar, Continuous neural network with
  windowed hebbian learning, Biological cybernetics (2015) 1--12.

\bibitem{bressloff2011spatiotemporal}
P.~C. Bressloff, Spatiotemporal dynamics of continuum neural fields, Journal of
  Physics A: Mathematical and Theoretical 45~(3) (2011) 033001.

\bibitem{amari1977dynamics}
S.-i. Amari, Dynamics of pattern formation in lateral-inhibition type neural
  fields, Biological cybernetics 27~(2) (1977) 77--87.

\bibitem{ermentrout1998neural}
B.~Ermentrout, Neural networks as spatio-temporal pattern-forming systems,
  Reports on progress in physics 61~(4) (1998) 353.

\bibitem{gerstner2002mathematical}
W.~Gerstner, W.~M. Kistler, Mathematical formulations of hebbian learning,
  Biological cybernetics 87~(5-6) (2002) 404--415.

\bibitem{chance2002gain}
F.~S. Chance, L.~Abbott, A.~D. Reyes, Gain modulation from background synaptic
  input, Neuron 35~(4) (2002) 773--782.

\bibitem{salinas2000gain}
E.~Salinas, P.~Thier, Gain modulation-a major computational principle of the
  central nervous system, Neuron 27~(1) (2000) 15--21.

\bibitem{chance2011gain}
F.~S. Chance, Gain modulation and stability in neural networks, Computational
  Neuroscience in Epilepsy (2011) 155.

\bibitem{stead2010microseizures}
M.~Stead, M.~Bower, B.~H. Brinkmann, K.~Lee, W.~R. Marsh, F.~B. Meyer, B.~Litt,
  J.~Van~Gompel, G.~A. Worrell, Microseizures and the spatiotemporal scales of
  human partial epilepsy, Brain (2010) awq190.

\bibitem{potthast2010existence}
R.~Potthast, P.~Beim~Graben, Existence and properties of solutions for neural
  field equations, Mathematical Methods in the Applied Sciences 33~(8) (2010)
  935--949.

\bibitem{faugeras2008absolute}
O.~Faugeras, F.~Grimbert, J.-J. Slotine, Absolute stability and complete
  synchronization in a class of neural fields models, SIAM Journal on Applied
  Mathematics 69~(1) (2008) 205--250.

\bibitem{oleynik2013properties}
A.~Oleynik, A.~Ponosov, J.~Wyller, On the properties of nonlinear nonlocal
  operators arising in neural field models, Journal of Mathematical Analysis
  and Applications 398~(1) (2013) 335--351.

\bibitem{da2012properties}
S.~H. Da~Silva, Properties of an equation for neural fields in a bounded
  domain, Electronic Journal of Differential Equations 2012~(42) (2012) 1--9.

\bibitem{da2014asymptotic}
S.~H. da~Silva, M.~B. Silva, Asymptotic behavior of neural fields in an
  unbounded domain, Differential Equations and Dynamical Systems (2014) 1--15.

\bibitem{riesz1955b}
F.~Riesz, S.~Nagy, B.(1990). functional analysis, Dover Publications, Inc., New
  York. First published in 3~(6) (1955) 35.

\bibitem{sun2005mercer}
H.~Sun, Mercer theorem for rkhs on noncompact sets, Journal of Complexity
  21~(3) (2005) 337--349.

\bibitem{Jimenez2017}
A.~Jimenez-Rodriguez, J.~Cordero-Ceballos, C.~Vargas, N.~Sanchez, Heterogeneous
  gain distributions in neural networks i:the stationary case, (Submitted).

\end{thebibliography}
\end{document}